\documentclass[]{imsart}

\RequirePackage[OT1]{fontenc}
\RequirePackage{amsthm,amsmath,amsfonts,amssymb}
\RequirePackage[colorlinks,citecolor=blue,urlcolor=blue]{hyperref}
\RequirePackage{natbib}
\usepackage{graphicx}
\usepackage{bm}
\usepackage{multirow}

\startlocaldefs
\numberwithin{equation}{section}
\theoremstyle{plain}
\newtheorem{theorem}{Theorem}[section]
\endlocaldefs

\theoremstyle{plain}

\newtheorem{lemma}{Lemma}[section]

\newtheorem{remark}{Remark}[section]

\newtheorem{example}{Example}[section]

\newtheorem{ass}{Assumption}[section]

\newcommand{\argmax}{\mathop{\rm arg~max}\limits}
\newcommand{\argmin}{\mathop{\rm arg~min}\limits}

\providecommand{\skakko}[1]{\left(#1\right)}
\providecommand{\mkakko}[1]{\left\{#1\right\}}
\providecommand{\lkakko}[1]{\left[#1\right]}

\allowdisplaybreaks

\begin{document}

\begin{frontmatter}
\title{A test for counting sequences\\ of integer-valued autoregressive models}
\runtitle{A test for counting sequences of INAR model}

\begin{aug}
\author{\fnms{Yuichi} \snm{Goto}\thanksref{t1}\ead[label=e1]{yuichi.goto@math.kyushu-u.ac.jp}}

\address{Faculty of Mathematics, Kyushu University,
\printead{e1}}

\and
\author{\fnms{Kou} \snm{Fujimori}\thanksref{t1}\ead[label=e2]{kfujimori@shinshu-u.ac.jp}}
\address{
Department of Applied Economics, Faculty of Economics and Law, Shinshu University,
\printead{e2}}

\thankstext{t1}{This research was supported by JSPS Grant-in-Aid for Early-Career Scientists JP23K16851 (Y.G.) and JP21K13271 (K.F.).
}
\runauthor{Y.Goto and K.Fujimori}

\end{aug}

\begin{abstract}
The integer autoregressive (INAR) model is one of the most commonly used models in nonnegative integer-valued time series analysis and is a counterpart to the traditional autoregressive model for continuous-valued time series. 
To guarantee the integer-valued nature,  
the binomial thinning operator or more generally the generalized Steutel and van Harn operator is used to define the INAR model.
However, the distributions of the counting sequences used in the operators have been determined by the preference of analyst without statistical verification so far.
In this paper, we propose a test based on the mean and variance relationships for distributions of counting sequences and a disturbance process to check if the operator is reasonable. We show that our proposed test has asymptotically correct size and is consistent.
Numerical simulation is carried out to evaluate the finite sample performance of our test. 
As a real data application, we apply our test to the monthly number of anorexia cases in animals submitted to animal health laboratories in New Zealand and we conclude that binomial thinning operator is not appropriate.
\end{abstract}


\tableofcontents
\end{frontmatter}

\section{Introduction}
Non-negative integer-valued time series are ubiquitous, and the analysis of this type of data has received much attention over the past two decades.  Applications of this research include the analysis of the number of infected people (\citealp{sp11}, \citealp{Xuetal23}), crimes (\citealp{cs17}, \citealp{bh19}), and so on. 
A review of recent developments in this area can be found in \cite{davisetal21}.

Integer autoregressive (INAR) models and integer-valued generalized autoregressive conditional heteroskedastic (INGARCH) models are popular models for non-negative integer-valued time series. INGARCH was introduced by \cite{flo06} and the conditional expectation of the process has autoregressive structures.
On the other hand, INAR model was proposed by \cite{mckenzie85} and \cite{aa87}, and is a natural extension of the classical autoregressive model for continuous-valued time series to count time series.
To ensure that the process takes non-negative integer values, the binomial thinning operator is often used in the model. 
The binomial thinning operator is defined as the sum of independent Bernoulli random variables called counting sequences.
\cite{latour97} extended the binomial thinning operator to the generalized Steutel and van Harn operator in multivariate settings, which allows the use of non-binomial counting sequences in the operator.
\cite{zj03} suggested an extended thinning operator by using the distribution including binomial distribution as a special case.
\cite{weib08} defined the new operator based on the convolution of the Bernoulli and the geometric distributions. 
\cite{rbn09} and \cite{yangetal19} studied the operator based on geometric and generalized Poisson distributions, respectively.
 \cite{ab19}  advocated the thinning operator based on a distribution with a linear fractional probability generating function.

Many operators have been proposed so far. However, operators have been completely determined by the analyst without statistical guarantees. In this paper, we propose a test for distributions of counting sequences and a distribution of the disturbance process. To construct the test, we make use of the important feature of the non-negative integer-valued distributions that the variance of the distribution takes the form of the function of its mean. 
Note that \cite{gf23} considered a test for conditional variances based on this feature. Their setting is essentially for INGARCH models and does not include INAR models since nuisance parameters in the conditional variance do not be allowed. To encompass INAR models, in {Appendix} we consider the integer-valued time series  $\{Z_t\}$ satisfying
\begin{align}\label{general_model}
{\rm E}\skakko{Z_t\mid \mathcal F_{t-1}} = m_t(\bm\mu_0)\text{ and }
{\rm Var}\skakko{Z_t\mid \mathcal F_{t-1}} = v_t(\bm\theta_0),
\end{align}
where 
$\mathcal F_{t-1}$ is the $\sigma$-field generated by $\{Z_{s},s\leq t-1\}$, for a known function $m$ on $[0,\infty)^\infty\times\mathbb R^d$ to $(\delta,+\infty)$ for some $\delta>0$, 
$$m_t(\bm\mu_0):=m(Z_{t-1},Z_{t-2},\ldots;\bm\mu_0),$$ for a known function $v$ on $[0,\infty)^\infty\times\mathbb R^{d^\prime}$ to $(\delta,+\infty)$,  $$v_t(\bm\theta_0):=v(Z_{t-1},Z_{t-2},\ldots;\bm\theta_0),$$  and $(\bm\mu_0^\top,\bm\theta_0^\top)^\top\in\mathbb R^{d+d^\prime}$ are unknown parameters, 
and establish the strong consistency and the asymptotic normality of the M-estimator of $\bm\mu_0$ and $\bm\theta_0$.
Since INAR models defined later in \eqref{INAR} have the conditional expectation and conditional variance of $\{Z_t\}$ given $\mathcal F_{t-1}$ of the form
\begin{align}\label{mvINAR}
m_t(\bm\mu_0) =
\sum_{i=1}^p \mu_iZ_{t-i} + \mu_\epsilon
\text{ and }
v_t(\bm\theta_0)=
\sum_{i=1}^p\sigma^2_iZ_{t-i}+\sigma_\epsilon^2,
\end{align}
where $\bm\mu_0:=(\mu_1,\ldots,\mu_p,\mu_\epsilon)^\top$ and $\bm\theta_0:=(\sigma_1^2,\ldots,\sigma_p^2,
\sigma_\epsilon^2)^\top$, the above setting includes INAR models.
In this article, we consider the conditional least squares estimator with a simple form of the asymptotic variance as a special case of the M-estimator. Based on the asymptotic normality of estimators for $\bm\mu_0$ and $\bm\theta_0$, we show that our proposed test has asymptotically size $\phi$ and is consistent.

As a related work, several goodness of fit tests for integer-valued time series have been developed. For instance, tests for the intensity function of INGARCH models were considered by \cite{n11}, \cite{fn13}, and \cite{ln13}. For the distributions of INGARCH and INAR, tests relied on probability generating functions were investigated by \cite{mk14}, \cite{hhm15}, \cite{schweer16} and \cite{hhm21}. Tests based on the Pearson statistic and the Stein-Chen identity were studied by \cite{weib18}) and \cite{awj22}, respectively. A test for overdispersion and zero inflation was proposed by \cite{whp19}.
However, tests related to operators for INAR models have not been considered so far.

The remainder of this paper is organized as follows.
In Section 2, we define generalized INAR models and formulate the test for distributions of counting sequences and a distribution of the disturbance process. A test statistic is also defined.
In Section 3.1,  we show the asymptotic null distribution of the proposed test statistic and the consistency of our test.
In Section 3.2, we show the strong consistency and asymptotic normality of the conditional least squares estimators.
In Section 4, we illustrate the finite sample performance of our test.
In Section 5, we apply our test to the monthly number of skin lesions and anorexia cases in animals submitted to animal health laboratories in New Zealand.
In Appendix, we derive sufficient conditions for the strong consistency and asymptotic normality of the M-estimator.

\section{Preliminary}\label{pre}
In this section, we define a generalized integer autoregressive process, a testing problem, and a test statistic.
Let $\{Z_t\}$ follow a generalized integer autoregressive process of order $p$ (GINAR($p$)) defined as
\begin{align}\label{INAR}
Z_t:=\sum_{i=1}^p \mu_i \circ Z_{t-i} + \epsilon_t,
\end{align}
where $\circ$ is so called the generalized Steutel and van Harn operator (see \cite{latour97}) given by $\mu_i\circ Z_{t-i}:=\sum_{j=1}^{Z_{t-i}}W_{i,t,j}$, $\{W_{i,t,j}\}$ is a counting sequence, which is an i.i.d.\  (with respect to $t$ and $j$) non-negative integer-valued random variable with mean $\mu_i$ and variance $\sigma_i^2$, 
$\{W_{i,t,j}\}$ is independent of $\{W_{i^\prime,t^\prime,j^\prime}\}$ for $i\neq i^\prime$, 
and $\{\epsilon_t\}$ is an i.i.d.\ non-negative integer-valued random variable with mean $\mu_\epsilon$ and variance $\sigma_\epsilon^2$, which is independent of $\{W_{i,t,j}\}$.

To ensure the strictly stationarity of $\{Z_t\}$, we suppose that $1-\mu_1z-\mu_2z^2-\cdots-\mu_pz^p\neq0$ for all $z\in\mathbb C$ such that $|z|\leq 1$. It is also known that $\{Z_t\}$ is ergodic. See \cite{latour97}.

In this paper, we are interested in the following hypothesis testing problem:
\begin{align}\nonumber
\tilde H_0&:\ \text{$(W_{1,t,j},\ldots,W_{p,t,j},\epsilon_t)$ follows a target joint distribution}\\\label{hypo}
&\text{v.s. $\tilde K_0$: $\tilde H_0$ does not hold.}
\end{align}

Before constructing the test statistic, we present some examples of a  counting sequence $\{W_{i,t,j}\}$.
\begin{example}\label{ex}{\rm 
The most popular choice of a counting sequence $\{W_{i,t,j}\}$ might be the Bernoulli random variable. In this case, the generalized Steutel and van Harn operator reduces to the binomial thinning. 

The second example  of $\{W_{i,t,j}\}$ is the extended thinning proposed by \cite{zj03}, which is defined as the generalized Steutel and van Harn operator with $\{W_{i,t,j}\}$ following a distribution with probability generating function 
$$
\phi_{\rm ZJ}(k):=\frac{(1-\mu_{\rm ZJ})+(\mu_{\rm ZJ}-\gamma_{\rm ZJ})k}{1-\mu_{\rm ZJ}\gamma_{\rm ZJ}-(1-\mu_{\rm ZJ})\gamma_{\rm ZJ} k} \quad \text{for $\mu_{\rm ZJ}\in[0,1]$ and $\gamma_{\rm ZJ}\in[0,1)$}.
$$ See also Remark A.9 of \cite{weib08}.
The distribution has mean $\mu_{\rm ZJ}$ and variance $\mu_{\rm ZJ}(1-\mu_{\rm ZJ}){(1+\gamma_{\rm ZJ})}/{(1-\gamma_{\rm ZJ})}$. 
According to \citet[Remark in p.239 and 240]{zj03}, the random variable can be generated by $B_{\rm ZJ}G_{\rm ZJ}$, where $B_{\rm ZJ}$ and $G_{\rm ZJ}$ follow the Bernoulli distribution with the success probability $(1-\gamma_{\rm ZJ})\mu_{\rm ZJ}/(1-\gamma_{\rm ZJ}\mu_{\rm ZJ})$ and the shifted geometric distribution with a parameter $1-(1-\mu_{\rm ZJ})\gamma_{\rm ZJ}/[\{(1-\mu_{\rm ZJ})\gamma_{\rm ZJ}\}+(1-\gamma_{\rm ZJ})]$, respectively, and $B_{\rm ZJ}$ is independent of $G_{\rm ZJ}$.

The third example of $\{W_{i,t,j}\}$ is the BerG$(\pi_{\rm BW},\xi_{\rm BW})$ distribution, where $\pi_{\rm BW}$ and $\xi_{\rm BW}$ are parameters satisfying $0<\pi_{\rm BW}<1$ and $\xi_{\rm BW}>0$, investigated by \cite{bw17}. The probability generating function, the mean, and the variance  of the distribution are given by 
$\phi_{\rm BW}(s):=\{1-\pi_{\rm BW}(1-s)\}/\{1+\xi_{\rm BW}(1-s)\}$, 
$\mu_{\rm BW}:=\pi_{\rm BW}+\xi_{\rm BW}$, and 
$\mu_{\rm BW}(1-\mu_{\rm BW}+2\xi_{\rm BW})$, respectively.
The random variable $B_{\rm BW}+G_{\rm BW}$, where $B_{\rm BW}$ and $G_{\rm BW}$ follow the Bernoulli distribution with the success probability $\pi_{\rm BW}$ and the geometric distribution with a parameter $1/(1+\xi_{\rm BW})$, respectively, follows the BerG$(\pi_{\rm BW},\xi_{\rm BW})$ distribution. 
\cite{ab19} considered a different parametrization from \cite{bw17} and elucidated that the family of distributions encompasses at least nine distributions as a special case including the second example above.

}
\end{example}

As mentioned in the introduction, the variances of many non-negative distributions can be written in term of their mean of the distributions. 
For example, the mean and variance of the binomial distribution with the number of trials $n=1$ and the success probability $p\in(0,1)$ take the form of $\mu_{\rm bin}:=p$, $\kappa_{\rm bin}(\mu_{\rm bin}):=\mu_{\rm bin}(1-\mu_{\rm bin})$, respectively. 
For the Poisson distribution with parameter $\lambda$, the mean and variance are given by $\mu_{\rm Pois}:=\lambda$ and $\kappa_{\rm Pois}(\mu_{\rm Pois}):=\mu_{\rm Pois}$. Similarly, for the negative binomial distribution with success probability $p$ and given parameter $r$, the mean and variance can be expressed as  $\mu_{\rm NB}:=r(1-p)/p$ and $\kappa_{NB}(\mu_{\rm NB}):=(\mu_{\rm NB}+r)\mu_{\rm NB}/r$.

 {Based on this fact, we rephrase the hypothesis defined in \eqref{hypo} as 
\begin{align}\nonumber
H_0:&\ \text{the variance of $(W_{1,t,j},\ldots,W_{p,t,j},\epsilon_t)$ takes the form of $\bm\kappa\skakko{\bm{{{\mu}}}_0}$}\\\label{hypo2}
&\quad \text{v.s. $K_0$: $H_0$ does not hold,}
\end{align}
where
$\bm\kappa\skakko{\bm{{{\mu}}}_0}:=\skakko{\kappa_1({
\mu}_1),\ldots,\kappa_p({\mu}_p),\kappa_\epsilon({\mu}_\epsilon)}^\top$ and $\kappa_1,\ldots,\kappa_p$, and $\kappa_\epsilon$ are the functions of means corresponding to variances of the target marginal distributions for $W_{1,t,j},\ldots,W_{p,t,j}$, and $\epsilon_t$, respectively.}
The hypothesis $H_0$ is a necessary condition for $\tilde H_0$.

\begin{remark} {{\rm 
In this paper, we specifically focused on the property of non-negative discrete random variables, where the variance is a function of its mean, and reframed the problem from the distributions of counting sequences to their second-order moments, leveraging the characteristic of non-negative discrete random variables.
Importantly, our proposed test does not have any power for counting processes that deviate from the target distribution but share the same first and second-order moments with the target distribution. To our best knowledge, non-negative discrete-valued random variables with separately parametrized mean and variance have not been introduced. Such variables would be highly desirable, though.}}
\end{remark}
To define a test statistic, we consider estimators $\bm{{\hat{\mu}}}_n$ and $\bm{{\hat{\theta}}}_n$ of $\bm{{{\mu}}}_0$ and $\bm{{{\theta}}}_0$ endowed with the strong consistency and the asymptotic normality with asymptotic variance $\bm V:=(\bm v_{ij})_{i,j=1,2}$, that is, 
\begin{align}\label{strong_consistent}
\begin{pmatrix}
\bm{\hat \mu}_n\\
\bm{\hat \theta}_n\\
\end{pmatrix}\rightarrow
\begin{pmatrix}
\bm\mu_0\\
\bm\theta_0\\
\end{pmatrix}\quad
\text{almost surely as $n\to\infty$}
\end{align}
and 
\begin{align}\label{asym_normal}
\sqrt n
\begin{pmatrix}
\bm{\hat \mu}_n - \bm\mu_0\\
\bm{\hat \theta}_n - \bm\theta_0\\
\end{pmatrix}\Rightarrow
N\skakko{\bm0, \bm V}
 \quad\text{as $n\to\infty$.}
\end{align}
The concrete examples of the estimators  {including the conditional least squares estimators and the M-estimators} will be presented in Section \ref{sec:ex_est}. 
Motivated by the mean and variance relationships, we consider the asymptotic distribution of $\bm\kappa(\bm{\hat{{\mu}}}_n)-\bm{\hat{{\theta}}}_n$. 
By Theorem \ref{thm3} and Proposition 6.4.3 of \cite{bd09}, we know that $\bm\kappa(\bm{{\hat{\mu}}}_n)-\bm{{\hat{\theta}}}_n$ converges in distribution to the $(p+1)$-dimensional centered normal distribution with variance
\begin{align}\label{W}
\bm W:=
\bm K \bm {v}_{11} \bm K 
-\bm K \bm {v}_{12} 
-\bm {v}_{21} \bm K +
\bm {v}_{22},
\end{align}
where $\bm K $ is a $(p+1)$-diagonal matrix whose $i$-th diagonal entry is 
$({\partial}/{\partial \mu_i})\kappa_i\skakko{\mu_i}$ for $i\in\{1,\ldots,p\}$ and $(p+1)$-th diagonal entry is $({\partial}/{\partial \mu_{\epsilon}})\kappa_\epsilon\skakko{\mu_\epsilon}$, provided that $\bm\kappa(\cdot)$ is continuously differentiable in a neighborhood of $\bm \mu_0$.
The quantities $\bm\kappa({\bm\mu}_0)$ and ${\bm\theta}_0$ correspond to the variance of $(W_{1,t,j},\ldots,W_{p,t,j},\epsilon_t)$ for the null hypothesis we wish to test and for the underlying distributions, respectively.

Thus, under the assumption that $\bm V$ and $\bm K $ are non-singular matrices we propose the following test statistic
\begin{align*}
T_n:=n\skakko{\bm\kappa\skakko{\bm{{\hat \mu}_n}}-\bm{{\hat \theta}_n}}^\top
{\widehat {\bm W}}^{-1}
\skakko{\bm\kappa\skakko{\bm{{\hat \mu}_n}}-\bm{{\hat \theta}_n}},
\end{align*}
where
$\widehat {\bm W}^{-1}$ is the inverse matrix of $\widehat{\bm W}$, $\widehat {\bm W}$ is the consistent estimator of $\bm W$. In Section \ref{sec:test}, we derive asymptotic properties of the test based on $T_n$.

\section{Main results}
\subsection{Asymptotic properties of the proposed test}\label{sec:test}

In this subsection, we derive the asymptotic distribution of $T_n$ and the consistency of the test based on $T_n$.

Based on the above discussion, the asymptotic null distribution of the test statistic is given in the following theorem.
\begin{theorem}
Suppose that estimators $\bm{{\hat{\mu}}}_n$ and $\bm{{\hat{\theta}}}_n$ of $\bm{{{\mu}}}_0$ and $\bm{{{\theta}}}_0$ satisfy the strong consistency and the asymptotic normality with asymptotic variance $\bm V:=(\bm v_{ij})_{i,j=1,2}$,
$\bm\kappa(\cdot)$ is continuously differentiable in a neighborhood of $\bm \mu_0$, 
matrices $\bm V$ and $\bm K $ are non-singular, and
$\widehat {\bm W}$ is the consistent estimator of $\bm W$.
Then, under the null hypothesis $H_0$, 
$T_n$ converges in distribution to the chi-square distribution with $p+1$ degrees of freedom as $n\to\infty$. 
\end{theorem}
\begin{proof}
The discussion in Section \ref{pre} and the assumption of this theorem yield the conclusion.
\end{proof}

\begin{remark} {\rm 
Examples of the estimators $\bm{{\hat{\mu}}}_n$ and $\bm{{\hat{\theta}}}_n$ of $\bm{{{\mu}}}_0$ and $\bm{{{\theta}}}_0$ satisfying the strong consistency and the asymptotic normality are given in Section \ref{sec:ex_est}. For the least squared estimator stated in Section \ref{sec:ex_est}, 
the moment conditions ${\rm E}W_{i,t,j}^8<\infty$ and ${\rm E}\epsilon_{t}^8<\infty$ suffice to construct the consistent estimator for  $\bm W$ (see also Remark \ref{rem_cons}).}
\end{remark}

Therefore, the test which rejects the null hypothesis $H_0$ whenever $T_n\geq \chi_{p+1}^2[1-\phi]$ has the asymptotic size $\phi$, where $\phi$ is a significance level and $\chi_{p+1}^2$ denotes the upper $\phi$-percentile of the chi-square distribution with ${p+1}$ degrees of freedom. 

The following theorem shows the consistency of the test.

\begin{theorem}
Suppose that estimators $\bm{{\hat{\mu}}}_n$ and $\bm{{\hat{\theta}}}_n$ of $\bm{{{\mu}}}_0$ and $\bm{{{\theta}}}_0$ satisfy the strong consistency and the asymptotic normality with asymptotic variance $\bm V:=(\bm v_{ij})_{i,j=1,2}$,
$\bm\kappa(\cdot)$ is continuously differentiable in a neighborhood of $\bm \mu_0$, 
matrices $\bm V$ and $\bm K $ are non-singular, and
$\widehat {\bm W}$ is the consistent estimator of $\bm W$. 
Then, the test based on $T_n$ is consistent, that is, (under the alternative hypothesis $K_0$,) the power of the test tends to one as $n\to\infty$. 
\end{theorem}

\begin{proof}
Denote $\bm\kappa_0$ and $\bm\kappa_1$ as the vector-valued functions of means  corresponding to variances of the target marginal distributions for $W_{1,t,j},\ldots,W_{p,t,j}$, and $\epsilon_t$ under the null hypothesis and under the alternative, respectively.
The consistency of $\widehat {\bm W}$ and $\bm{{\hat{\mu}}}_n$ and $\bm{{\hat{\theta}}}_n$ gives that, under the alternative hypothesis $K_0$,
\begin{align*}
&{\rm P}\skakko{\frac{T_n}{n}\geq\frac{\chi_{p+1}^2[1-\phi]}{n}}\\
=&
{\rm P}\Bigg(
\skakko{\bm\kappa_0\skakko{\bm{{ \mu}_0}}-\bm\kappa_1\skakko{\bm{{ \mu}_0}}+\bm\kappa_1\skakko{\bm{{ \mu}_0}}-\bm{{ \theta}_0}}^\top
{{\bm W}}^{-1}\\
&\quad\times\skakko{\bm\kappa_0\skakko{\bm{{ \mu}_0}}-\bm\kappa_1\skakko{\bm{{ \mu}_0}}+\bm\kappa_1\skakko{\bm{{ \mu}_0}}-\bm{{ \theta}_0}}\geq\frac{\chi_{p+1}^2[1-\phi]}{n}\Bigg)\\
=&
{\rm P}\Bigg(
\skakko{\bm\kappa_0\skakko{\bm{{ \mu}_0}}-\bm\kappa_1\skakko{\bm{{ \mu}_0}}}^\top
{{\bm W}}^{-1}
\skakko{\bm\kappa_0\skakko{\bm{{ \mu}_0}}-\bm\kappa_1\skakko{\bm{{ \mu}_0}}}\geq0\Bigg) +o_p(1).
\end{align*}
Here we used the relationship $\bm\kappa_1\skakko{\bm{{ \mu}_0}}=\bm{{ \theta}_0}$. 
The definition of the nonnegative-definite matrix yields 
$${\rm P}\Bigg(
\skakko{\bm\kappa_0\skakko{\bm{{ \mu}_0}}-\bm\kappa_1\skakko{\bm{{ \mu}_0}}}^\top
{{\bm W}}^{-1}
\skakko{\bm\kappa_0\skakko{\bm{{ \mu}_0}}-\bm\kappa_1\skakko{\bm{{ \mu}_0}}}\geq0\Bigg)=1.$$
\end{proof}

\subsection{Example of the estimators}\label{sec:ex_est}

In this subsection, we derive the example of the estimators $\bm{{\hat{\mu}}}_n$ and $\bm{{\hat{\theta}}}_n$ of $\bm{{{\mu}}}_0$ and $\bm{{{\theta}}}_0$ endowed with the strong consistency and the asymptotic normality.

The conditional least squares estimators for $\bm\mu$ and $\bm\theta$ are defined as
\begin{align*}
\bm{\hat \mu}_n
:=\argmin_{\bm\mu\in \bm H}\frac{1}{n}\sum_{t=1+p}^n(Z_t- \bm\mu^\top{\bm Y}_{t-1})^2
\end{align*} and
\begin{align*}
\bm{\hat \theta}_n
:=\argmin_{\bm\theta\in \bm \Theta}\frac{1}{n}\sum_{t=1+p}^n\skakko{(Z_t-\bm\mu^\top{\bm Y}_{t-1})^2-\bm{\theta}^\top{\bm Y}_{t-1}}^2,
\end{align*}
where ${\bm Y}_{t-1}:=(Y_{t-1},\ldots,Y_{t-p},1)^\top$, $\bm H$ is a compact subspace of $\bm{\mu}\in\mathbb R^{p+1}$ such that $1-\mu_1z-\mu_2z^2-\cdots-\mu_pz^p\neq0$ for all $z\in\mathbb C$ such that $|z|\leq 1$ and $\mu_i>0$ for any $i\in\{1,\ldots,p\}$ and $\bm \Theta$ is a compact subspace of $\bm\theta\in\mathbb R^{p+1}$ such that $\sigma_i^2>0$ for any $i\in\{1,\ldots,p\}$ (see \citealp{kn78}). 
The least squared estimator can be written explicitly as 
\begin{align*}
\bm{\hat \mu}_n
:=\skakko{\frac{1}{n}\sum_{t=1+p}^n{\bm Y}_{t-1}{\bm Y}_{t-1}^\top}^{-1}\frac{1}{n}\sum_{t=1+p}^n Z_t{\bm Y}_{t-1}
\end{align*}
and
\begin{align*}
\bm{\hat \theta}_n
:=\skakko{\frac{1}{n}\sum_{t=1+p}^n{\bm Y}_{t-1}{\bm Y}_{t-1}^\top}^{-1}\frac{1}{n}\sum_{t=1+p}^n \skakko{Z_t-\bm{\hat \mu}_n^\top{\bm Y}_{t-1}}^2{\bm Y}_{t-1}.
\end{align*}

The strong consistency and asymptotic normality of $\bm{\hat \mu}_n$ and $\bm{\hat \theta}_n$ are shown in the following theorem.

\begin{lemma}\label{thm3}
Suppose that ${\rm E}W_{i,t,j}^8<\infty$ and ${\rm E}\epsilon_{t}^8<\infty$.
Then, $\bm{\hat \mu}_n$ and $\bm{\hat \theta}_n$ tend almost surely to $\bm\mu_0$ and $\bm\theta_0$ as $n\to\infty$, respectively. Moreover, it holds that 
\begin{align*}
\sqrt n
\begin{pmatrix}
\bm{\hat \mu}_n - \bm\mu_0\\
\bm{\hat \theta}_n - \bm\theta_0\\
\end{pmatrix}\Rightarrow
N\skakko{\bm0, \bm{J^{-1}}\bm{I}\bm{J^{-1}}}
 \quad\text{as $n\to\infty$,}
\end{align*}
where $\bm {\mathcal O}_{p+1}$ is a $(p+1)\times(p+1)$ zero matrix, 
\begin{align*}
\bm{J}:=
\begin{pmatrix}
\bm{J}_m & \bm{\mathcal O}_{p+1} \\
\bm{\mathcal O}_{p+1} & \bm{J}_v\\
\end{pmatrix},
\text{ and }
\bm{I}:=
\begin{pmatrix}
\bm{I}_m & \bm{I}_{mv}\\
\bm{I}_{mv}^\top & \bm{I}_v\\
\end{pmatrix}
\end{align*}
with 
${\bm J_m}:=
{\rm E}\skakko{
{\bm Y}_{t-1}{\bm Y}_{t-1}^\top}$, 
${\bm J_v}:=
{\rm E}\skakko{
{\bm Y}_{t-1}{\bm Y}_{t-1}^\top}$, 
${\bm I_m}:=
{\rm E}\skakko{
\skakko{\bm{\theta}_0^\top{\bm Y}_{t-1}}
{\bm Y}_{t-1}{\bm Y}_{t-1}^\top},$
$\bm I_{mv}:={\rm E}\skakko{
(Z_t-{\bm\mu}_0^\top{\bm Y}_{t-1})^3
{\bm Y}_{t-1}{\bm Y}_{t-1}^\top}$, and\\ \noindent
$\bm I_v :={\rm E}\skakko{\skakko{(Z_t-{\bm\mu}_0^\top{\bm Y}_{t-1})^4-\skakko{\bm{\theta}_0^\top{\bm Y}_{t-1}}^2}{\bm Y}_{t-1}{\bm Y}_{t-1}^\top}$.

\end{lemma}
\begin{proof}
It is easy to check the conditions (B1)--(B12) and (C1)--(C13) in Appendix. Thus, Theorem \ref{thma1} in Appendix yields the conclusion.
\end{proof}
\begin{remark}\label{rem_cons}{\rm 
The conditions ${\rm E}W_{i,t,j}^8<\infty$ and ${\rm E}\epsilon_{t}^8<\infty$ ensure that ${\rm E}Z_{t}^8<\infty$ (See Theorem 2.2 of \citealt{dvrw08}). The condition ${\rm E}Z_{t}^8<\infty$ is used when we show $\bm I_{v}$ is finite. To see this, it suffices to show that, for any $s,k\in\{1,\ldots,{p}\}$, 
$${\rm E}\skakko{\skakko{(Z_t-{\bm\mu}_0^\top{\bm Y}_{t-1})^4-\skakko{\bm{\theta}_0^\top{\bm Y}_{t-1}}^2}Z_{t-s}Z_{t-k}}<\infty$$
 since the proofs of 
$${\rm E}\skakko{\skakko{(Z_t-{\bm\mu}_0^\top{\bm Y}_{t-1})^4-\skakko{\bm{\theta}_0^\top{\bm Y}_{t-1}}^2}Z_{t-s}}<\infty$$ and $${\rm E}\skakko{\skakko{(Z_t-{\bm\mu}_0^\top{\bm Y}_{t-1})^4-\skakko{\bm{\theta}_0^\top{\bm Y}_{t-1}}^2}}<\infty$$ are similar. 
By the Cauchy--Schwarz inequality and the definition of the convex function, we obtain
\begin{align*}
&\mkakko{{\rm E}\skakko{\skakko{(Z_t-{\bm\mu}_0^\top{\bm Y}_{t-1})^4-\skakko{\bm{\theta}_0^\top{\bm Y}_{t-1}}^2}Z_{t-s}Z_{t-k}}}^{2}\\
\leq&
{\rm E}\skakko{(Z_t-{\bm\mu}_0^\top{\bm Y}_{t-1})^4-\skakko{\bm{\theta}_0^\top{\bm Y}_{t-1}}^2}^2
{\rm E}\skakko{Z_{t-s}^2Z_{t-k}^2}\\
\leq&
2{\rm E}\skakko{(Z_t-{\bm\mu}_0^\top{\bm Y}_{t-1})^8+\skakko{\bm{\theta}_0^\top{\bm Y}_{t-1}}^4}
\mkakko{{\rm E}\skakko{Z_{t-s}^4}}^{1/2}\mkakko{{\rm E}\skakko{Z_{t-k}^4}}^{1/2}\\
\leq&
{\rm E}\skakko{2^8Z_t^8+2^8({\bm\mu}_0^\top{\bm Y}_{t-1})^8+2\skakko{\bm{\theta}_0^\top{\bm Y}_{t-1}}^4}
\mkakko{{\rm E}\skakko{Z_{t-s}^4}}^{1/2}\mkakko{{\rm E}\skakko{Z_{t-k}^4}}^{1/2},
\end{align*}
which is finite since the repeated applications of the definition of the convex function yields that
${\rm E}\skakko{({\bm\mu}_0^\top{\bm Y}_{t-1})^8}$ is finite.

The estimator of $\bm{J^{-1}}\bm{I}\bm{J^{-1}}$ can be constructed by replacing ${\bm\mu}_0$ and ${\bm\theta}_0$ with $\hat{\bm\mu}_n$ and $\hat{\bm\theta}_n$, respectively, using empirical sums instead of expectations. Under the condition ${\rm E}Z_{t}^8<\infty$, the consistency of the estimator of $\bm{J^{-1}}\bm{I}\bm{J^{-1}}$can be shown by the uniform law of large numbers (see, e.g., \citealt[Theorem 7.3.1 (ii)]{nishiyama21}).}
\end{remark}

\begin{remark}
{\rm 
In Appendix, we derive the asymptotic normality of the M-estimator for the general model defined in \eqref{general_model}. Since $m_t(\bm\mu)$ and $v_t(\bm\mu,\bm\theta)$ include unobservable variables $Z_{0},Z_{-1},\ldots$, define computable quantities
$$\tilde m_t(\bm\mu):=m(Z_{t-1},Z_{t-2},\ldots,Z_1,\bm x_0;\bm\mu)$$ and 
$$\tilde v_t(\bm\theta):=v(Z_{t-1},Z_{t-2},\ldots,Z_1,\bm x_0;\bm\theta)$$ for some constant $\bm x_0$. 
The M-estimator is defined as
\begin{align*}
\bm{\check\mu}_n
:=\argmax_{\bm\mu\in \bm H}\frac{1}{n}\sum_{t=1}^n\ell_m(Z_t,\tilde m_t(\bm\mu))
\end{align*} and
\begin{align*}
\bm{\check\theta}_n
:=\argmax_{\bm\theta\in \bm \Theta}\frac{1}{n}\sum_{t=1}^n\ell_v(Z_t,\tilde m_t(\bm{\check\mu}_n),\tilde v_t(\bm\theta)),
\end{align*}
where $\ell_m$ and $\ell_v$ are some estimating functions.

 The asymptotic variance of  the M-estimator  is given by $\tilde{\bm V}:=(\tilde
{\bm v}_{ij})_{i,j=1,2}$, where
\begin{align*}
\tilde{\bm v}_{11}:=&{\bm {\tilde J}_m^{-1}}{\bm {\tilde I}_m}{\bm {\tilde J}_m^{-1}},\quad
\tilde{\bm v}_{12}:=
{\bm {\tilde J}_{m}^{-1}}({\bm {\tilde I}_{mv}}-{\bm {\tilde I}_m}{\bm {\tilde J}_m^{-1}}{\bm {\tilde J}_{vm}^\top}){\bm {\tilde J}_{v}^{-1}},\quad
\tilde{\bm v}_{21}:=\tilde{\bm v}_{12}^\top,\\
\tilde{\bm v}_{22}:=&{\bm {\tilde J}_{v}^{-1}}(
{\bm {\tilde I}_{v}}
+{\bm {\tilde J}_{vm}}{\bm {\tilde J}_{m}^{-1}}{\bm {\tilde I}_{m}}{\bm {\tilde J}_{m}^{-1}}{\bm {\tilde J}_{vm}^{\top}}
-{\bm {\tilde I}_{mv}^{\top}}{\bm {\tilde J}_{m}^{-1}}{\bm {\tilde J}_{vm}^{\top}}
-{\bm {\tilde J}_{vm}}{\bm {\tilde J}_{m}^{-1}}{\bm {\tilde I}_{mv}}
){\bm {\tilde J}_{v}^{-1}}
\end{align*}
with
\begin{align*}
{\bm {\tilde I}_m}:=&{\rm E}\skakko{
\frac{\partial}{\partial\bm\mu }\ell_m(Z_t, m_t(\bm\mu_0))
\frac{\partial}{\partial\bm\mu^\top}\ell_m(Z_t, m_t(\bm\mu_0))
},\\
{\bm {\tilde J}_m}:=&{\rm E}\skakko{\frac{\partial^2}{\partial\bm\mu \partial\bm\mu^\top}\ell_m(Z_t, m_t(\bm\mu_0))},\\
\bm {\tilde I}_v :=&{\rm E}\skakko{\frac{\partial}{\partial \bm \theta}\ell_v(Z_t, m_t(\bm{\mu}_0), v_t(\bm{\theta}_0))\frac{\partial}{\partial \bm \theta^\top}\ell_v(Z_t, m_t(\bm{\mu}_0), v_t(\bm{\theta}_0))},\\
\bm{{\tilde J}}_{v}:=&{\rm E}\skakko{\frac{\partial^2}{\partial \bm \theta \partial \bm \theta^\top}\ell_v(Z_t, m_t(\bm{\mu}_0), v_t(\bm{\theta}_0))},\\
\bm {\tilde I}_{mv}:=&{\rm E}\skakko{\frac{\partial}{\partial \bm \mu}\ell_m(Z_t, m_t(\bm{\mu}_0))\frac{\partial}{\partial \bm \theta^\top}\ell_v(Z_t, m_t(\bm{\mu}_0), v_t(\bm{\theta}_0))},\\
\text{and }\bm{{\tilde J}}_{vm}:=&{\rm E}\skakko{\frac{\partial^2}{\partial \bm \theta \partial \bm \mu^\top}\ell_v(Z_t, m_t(\bm{\mu}_0), v_t(\bm{\theta}_0))}.
\end{align*}
The asymptotic variance of $\bm{\check \theta}_n$ is complicated due to the fact that $\ell_v$ includes $\tilde m_t(\bm{\check\mu}_n)$.
In the proof, it can be shown that
\begin{align*}
&\begin{pmatrix}
\sqrt n(\bm{\check \mu}_n-\bm{ \mu}_0)\\
\sqrt n(\bm{\check\theta}_n-\bm{ \theta}_0)
\end{pmatrix}\\
=&
\begin{pmatrix}
\bm {\tilde J}_m^{-1}&\bm {\mathcal O}_{p+1}\\
-\bm{\tilde J}_{v}^{-1}\bm{\tilde J}_{vm}\bm {\tilde J}_m^{-1}&-\bm{\tilde  J}_{v}^{-1}
\end{pmatrix}
\begin{pmatrix}
\frac{1}{\sqrt n}\sum_{t=1}^n 
\frac{\partial}{\partial\bm\mu}\ell_m(Z_t,m_t(\bm\mu_0))
\\
\frac{1}{\sqrt n}\sum_{t=1}^n
\frac{\partial}{\partial \bm \theta}\ell_v(Z_t, m_t(\bm{\mu}_0), v_t(\bm\theta_0))
\end{pmatrix}+o_p(1),
\end{align*}
where $\bm {\mathcal O}_{p+1}$ is a $(p+1)\times(p+1)$ zero matrix, and the martingale central limit theorem yields the asymptotic normality.

On the other hand, the asymptotic variance of the least squares estimators is considerably simple due to the fact that $${\rm E}\skakko{\frac{\partial^2}{\partial \bm \theta \partial \bm \eta^\top}((Z_t- m_t(\bm{\mu}_0))^2-v_t(\bm{\theta}_0))^2}=\bm{\mathcal O}_{p+1},$$ 
which corresponds to $\bm{{\tilde J}}_{vm}:=\bm{\mathcal O}_{p+1}$. 

One may consider that the alternative definition of the least squares estimator as
\begin{align*}
\bm{\breve \theta}_n
:=\argmin_{\bm\theta\in \bm \Theta}\frac{1}{n}\sum_{t=1}^n(Z_t^2- m_t(\bm{\hat\mu}_n)^2-v_t(\bm{\theta}))^2.
\end{align*}
However in this case, we have
\begin{align*}
&{\rm E}\skakko{\frac{\partial^2}{\partial \bm \theta \partial \bm \eta^\top}(Z_t^2- m_t(\bm{\mu}_0)^2-v_t(\bm{\theta}_0))^2}\\
=&
{\rm E}\skakko{
4m_t(\bm{\mu}_0)
\frac{\partial}{\partial \bm \theta }v_t(\bm{\theta}_0)
\frac{\partial}{\partial \bm \mu^\top}  m_t(\bm{\mu}_0))
},
\end{align*}
which is not equal to a zero matrix. 
Therefore, our choise for $\ell_v$ is essential to obtain a simple form for the asymptotic variance of the estimator.}
\end{remark}

\begin{remark}{\rm 
 {It is worth to highlight that we established the asymptotic normality of the estimator for the variance of counting sequences. Therefore, it is reasonable to expect that the eighth order moments of the process to estimate the asymptotic variance of the variance for counting sequences.
In that case one still think that the assumption of eighth order moment is strong, we suggest to apply the weighted least square estimator (e.g. Section 3.5 of \cite{ac21} and Remark below Theorem 2 of \citealt{kt11}), which has $\sqrt n$-convergence rate, to relax the moment assumption.}

 {On the other hand,  the technique of \cite{c86} or \cite{ps94} might be applied to relax the moment assumption. Nonetheless, a fundamental trade-off exists in convergence rate. Specifically, their estimator converges at some rate slower than $\sqrt n$.}}
\end{remark}

\begin{remark}{\rm 
{Our approach can be applied to the integer-valued bilinear process defined as 
\begin{align*}
Z_t:=\sum_{i=1}^p \alpha_i \circ Z_{t-i} 
+\sum_{j=1}^q \beta_j \circ \epsilon_{t-i}
+\sum_{k=1}^m\sum_{\ell=1}^n \gamma_{k,\ell} \circ (X_{t-k}\epsilon_{t-\ell})
+ \epsilon_t,
\end{align*}
the  integer-valued ARCH process (see \citealt{dlo06}, \citealt{dvrw08}, \citealt{lt09}, \citealt{dfl12}), the signed INAR model  (see \citealt{kt11}, \citealt{lcz23}, and references therein). 
}}
\end{remark}

\section{Numerical study}
In this section, we investigate the finite sample performance of our test.
We consider INAR(1) model whose counting sequence and disturbance process follow BerG$(\pi_{\rm BW},\xi_{\rm BW})$ and Poisson distributions with parameter $\lambda=1$, respectively. See Example \ref{ex}.

We set the null hypothesis as the counting sequence following the Bernoulli distribution, which corresponds to the case $\xi_{\rm BW}=0$, and the disturbance process following the Poisson distribution. 
For this null hypothesis, the functions $\kappa_1(\mu_1)$ and $\kappa_\epsilon(\mu_\epsilon)$ are given by
$\kappa_1(\mu_1):=\mu_1(1-\mu_1)$ and $\kappa_\epsilon(\mu_\epsilon):=\mu_\epsilon$, respectively, where
$\mu_1:=\pi_{\rm BW}$ and $\lambda:=\mu_\epsilon$.
Then, the matrix $\bm K $ in $\bm W$ is given by
$\bm K:=\begin{pmatrix}
1-2\mu_1&0\\
0&1
\end{pmatrix}$.

Let the significance level of the test be $\phi:=0.05$, the number of iterations $R:=1000$, the time series length $n\in\{500,1000, 2000\}$, burn-in period 1000.
We use the parameters $(\pi_{\rm BW},\xi_{\rm BW})$ on the sets $\{0.2,0.3,\ldots,0.8\}\times\{0\}$ for the null hypothesis and
$\{0.2,0.3,\ldots,0.6\}\times\{0.05, 0.1, \ldots, 0.3\}$ for the alternative hypothesis. Note that the parameters $\pi_{\rm BW}$ and $\xi_{\rm BW}$ need to satisfy the condition $\pi_{\rm BW}+\xi_{\rm BW}<1$ to ensure stationarity of the process. 
The least squares estimators are used to estimate unknown parameters. 

We generate INAR(1) process with length $n$ and apply our test. Then, we iterate 1000 times to calculate empirical size or power. 
Table \ref{tb1} shows  the empirical size of the test. 
We observe that our test has reasonable size overall although our test has slightly over-rejection for the large values of $\pi_{\rm BW}$. 
This corresponds to the fact that when $\pi_{\rm BW}$ takes large values, $\pi_{\rm BW}+\xi_{\rm BW}$ is close to the boundary of the stationary region of INAR(1).

Table \ref{tb2} presents the empirical power of the test. As $n$ or $\xi_{\rm BW}$ gets larger, as the power of the test increases. As expected, it is more difficult to reject the null hypothesis when $\xi_{\rm BW}$ is small, which is the case that the null and alternative hypotheses are close.

\begin{table}
  \begin{center}
    \caption{The empirical size at the nominal size $0.05$}
\label{tb1}
    \begin{tabular}{cccc}
$\pi_{\rm BW}$&\multicolumn{3}{c}{$n$}\\\cline{2-4}
 &500&1000&2000\\\hline
0.2&0.066 &0.070 &0.059\\
0.3&0.070 &0.062 &0.050\\
0.4&0.093 &0.063 &0.057\\
0.5&0.087 &0.065 &0.056\\
0.6&0.086 &0.054 &0.059\\
0.7&0.076 &0.070 &0.070\\
0.8&0.090 &0.075 &0.075\\
    \end{tabular}
 
  \end{center}
\end{table}

\begin{table}
\centering
\caption{The empirical power at the nominal power $0.05$}
\label{tb2}
\begin{tabular}{cccccccc}
\multicolumn{1}{c}{$\pi_{\rm BW}$} & \multicolumn{1}{c}{$n$} &\multicolumn{6}{c}{$\xi_{\rm BW}$} \\ \cline{3-8}
& & \multicolumn{1}{c}{0.05} & \multicolumn{1}{c}{0.1} & \multicolumn{1}{c}{0.15} & \multicolumn{1}{c}{0.2} & \multicolumn{1}{c}{0.25}& \multicolumn{1}{c}{0.3}\\ \cline{1-8}
\multirow{3}{*}{0.2} 
& 500 &0.061 & 0.084 & 0.180 & 0.389 & 0.643 & 0.903 \\ 
& 1000 &0.049 & 0.128 & 0.411 & 0.758 & 0.968 & 0.999 \\ 
& 2000 &0.088 & 0.327 & 0.775 & 0.978 & 1.000 & 1.000 \\ \cline{1-8}
\multirow{3}{*}{0.3} 
& 500 &0.047 & 0.097 & 0.302 & 0.603 & 0.865 & 0.983 \\ 
& 1000 &0.073 & 0.215 & 0.640 & 0.933 & 0.998 & 1.000 \\ 
& 2000 &0.111 & 0.510 & 0.948 & 1.000 & 1.000 & 1.000 \\ \cline{1-8}
\multirow{3}{*}{0.4} 
& 500 &0.064 & 0.165 & 0.486 & 0.806 & 0.971 & 1.000 \\ 
& 1000 &0.078 & 0.401 & 0.840 & 0.994 & 1.000 & 1.000 \\
& 2000 &0.178 & 0.748 & 0.990 & 1.000 & 1.000 & 1.000 \\ \cline{1-8}
\multirow{3}{*}{0.5} 
& 500 &0.072 & 0.303 & 0.694 & 0.949 & 0.999 & 0.993 \\ 
& 1000 & 0.129 & 0.598 & 0.970 & 1.000 & 1.000 & 0.998 \\ 
& 2000 &0.289 & 0.912 & 1.000 & 1.000 & 1.000 & 1.000 \\ \cline{1-8}
\multirow{3}{*}{0.6} 
& 500 &0.086 & 0.418 & 0.870 & 0.981 & 0.974 & 0.951 \\ 
& 1000 &0.190 & 0.785 & 0.993 & 0.999 & 0.983 & 0.957 \\ 
& 2000 &0.445 & 0.987 & 1.000 & 1.000 & 0.998 & 0.973 \\ 
\end{tabular}
\end{table}

\section{Empirical application}
In this section, we apply our method to the monthly number of skin lesions and anorexia cases in animals submitted to animal health laboratories in New Zealand (84 observations were collected from January 2003 to December 2009). These data can be found in Tables 3 and 4 of \cite{mbs18}.  \cite{jjl12} and \cite{mbs18} fitted these datasets to binomial thinning base INAR(1) model with zero inflated Poisson innovations and innovations induced by the Poisson--Lindley distribution, respectively.
Our objective of this section is to confirm if the binomial thinning operator is appropriate for the datasets.

The estimated values of the least squares estimators $\bm{{\hat \mu}}_{n}=({\hat{ \mu}}_{1},{{\hat \mu}}_{\epsilon})$ and $\bm{{\hat \theta}}_{n}=({\hat{ \sigma}}^2_{1},{\hat{ \sigma}}^2_{\epsilon})$ are $(0.325,0.964)$ and $(0.841,1.97)$ for skin-lesions and $(0.680,0.263)$ and $(2.40,0.102)$ for anorexia,  respectively.
The estimation shows that the counting series and disturbance process have overdispersion property for skin-lesions and overdispersion and underdispersion properties for anorexia, respectively.

We consider the null hypothesis that the datasets follow INAR(1) model with the binomial thinning operator. Our test only for the thinning operator (not for both thinning operator and innovation) gives p-values 0.182 for skin-lesions and 0.00277 for anorexia. We conclude there is no evidence to reject the null hypothesis for skin-lesions but that binomial thinning operator is not appropriate for anorexia. Therefore, we suggest that the counting series should be chosen carefully, rather than using the binomial thinning operator that has been used as customary.

\appendix


\section{General model and estimation}\label{suppl1}
In this section, we introduce the general model including the INAR($p$) model and show the strong consistency and the asymptotic normality of the M-estimator.

Let $\{Z_t\}$ be integer-valued time series satisfying
\begin{align*}
{\rm E}\skakko{Z_t\mid \mathcal F_{t-1}} = m_t(\bm\eta_0)\text{ and }
{\rm Var}\skakko{Z_t\mid \mathcal F_{t-1}} = v_t(\bm\theta_0),
\end{align*}
where 
$\mathcal F_{t-1}$ is the $\sigma$-field generated by $\{Z_{s},s\leq t-1\}$, for a known function $m$ on $[0,\infty)^\infty\times\mathbb R^d$ to $(\delta,+\infty)$ for some $\delta>0$, 
$$m_t(\bm\eta_0):=m(Z_{t-1},Z_{t-2},\ldots;\bm\eta_0),$$ for a known function $v$ on $[0,\infty)^\infty\times\mathbb R^{d^\prime}$ to $(\delta,+\infty)$,  $$v_t(\bm\theta_0):=v(Z_{t-1},Z_{t-2},\ldots;\bm\theta_0),$$  and $(\bm\eta_0^\top,\bm\theta_0^\top)^\top\in\mathbb R^{d+d^\prime}$ is unknown parameter. 

Since $m_t(\bm\eta)$ and $v_t(\bm\theta)$ include unobservable variables $Z_{0},Z_{-1},\ldots$, define 
$$\tilde m_t(\bm\eta):=m(Z_{t-1},Z_{t-2},\ldots,Z_1,\bm x_0;\bm\eta)$$ and 
$$\tilde v_t(\bm\theta):=v(Z_{t-1},Z_{t-2},\ldots,Z_1,\bm x_0;\bm\theta)$$ for some constant $\bm x_0$. We assume the following condition throughout this Appendix.
\begin{ass}{\rm 
(A) $\{Z_t\}$ is strictly stationary and ergodic.}
\end{ass}

The M-estimator of $\bm{\eta}_0$ is defined as 
\begin{align*}
\bm{\hat \eta}_n
:=\argmax_{\bm\eta\in \bm H}\tilde L_{n,m}(\bm\eta),
\end{align*}
where 
$\tilde L_{n,m}(\eta):= \frac{1}{n}\sum_{t=1}^n \ell_m(Z_t,\tilde m_t(\bm\eta))$ and 
$\ell_m$ is a measurable function.

We assume the next condition to prove the strong consistency and the asymptotic normality of $\bm{\hat \eta}_n$.
\begin{ass}{\rm 
(B1) It holds that
$$\sup_{\bm \eta\in \bm H}
\left|\ell_m(Z_t,\tilde m_t(\bm\eta))-\ell_m(Z_t, m_t(\bm\eta))\right|\to0 \quad \text{a.s. as $t\to\infty$.}$$

\noindent
(B2) The function $m$ is almost surely continuous with respect to $\bm \eta$ and 
$\ell_m(\cdot,\cdot)$ is almost surely continuous  with respect to its second component.\\

\noindent
(B3) The expectation of $\ell_m(Z_t, m_t(\bm\eta_0))$ is finite.\\

\noindent
(B4) The expectation of $\ell_m(Z_t, m_t(\bm\eta_0))$ has a unique maximum at $\bm\eta_0$.\\

\noindent
(B5) The parameter space $\bm H$ is compact.\\

\noindent
(B6) The function $m$ is twice continuously differentiable with respect to $\bm \eta$ and 
$\ell_m(\cdot,\cdot)$ is twice continuously differentiable  with respect to its second component.\\

\noindent
(B7) It holds that
$$\sup_{\bm \eta\in \bm H}
\left\|\frac{\partial}{\partial \bm \eta}\ell_m(Z_t,\tilde m_t(\bm\eta))-\frac{\partial}{\partial \bm \eta}\ell_m(Z_t, m_t(\bm\eta))\right\|=O(t^{-\frac{1}{2}-\delta}) \quad \text{a.s. as $t\to\infty$.}$$\\

\noindent
(B8) The true parameter $\bm \eta_0$ belongs to the interior of $\bm H$.\\

\noindent
(B9) There exists a neighborhood $B(\bm \eta_0)$ of $\bm \eta_0$ such that
\begin{align*}
{\rm E}\skakko{
\sup_{\bm \eta \in B(\bm \eta_0)}\left\|
\frac{\partial^2}{\partial\bm\eta\partial\bm\eta^\top}
\ell_m(Z_t, m_t(\bm\eta))
\right\|}<\infty.
\end{align*}

\noindent
(B10) It holds that ${\rm E}\skakko{\ell_m^\prime(Z_t,m_t(\bm\eta_0))\mid \mathcal F_{t-1}}=0$ a.s., where $\ell_m^\prime(\cdot,\cdot)$ is the first derivative of $\ell_m(\cdot,\cdot)$ with respect to its second component.\\

\noindent
(B11) The following holds true: ${\rm E}\skakko{\ell_m^{\prime\prime}(Z_t,m_t(\bm\eta_0))\mid \mathcal F_{t-1}}$ is almost surely negative, where $\ell_m^{\prime\prime}(\cdot,\cdot)$ is the second derivative of $\ell_m(\cdot,\cdot)$ with respect to its second component, and if 
${\bm s}^\top\frac{\partial}{\partial\bm\eta}m_t(\bm\eta_0)=0$, then 
$\bm s=\bm0$.\\

\noindent
(B12) 
It holds that
$${\rm E}\left\|
\frac{\partial}{\partial \bm \eta}\ell_m(Z_t, m_t(\bm\eta_0))
\right\|^{2+\delta}<\infty.$$}
\end{ass}

The following lemma is due to Theorems 1 and 2 of \cite{gf23}.
\begin{lemma}
Under Assumptions (A) and (B1)--(B5),  $\bm{\hat \eta}_n$ tends almost surely to $\bm\eta_0$ as $n\to\infty$. Moreover, under Assumptions (A) and (B1)--(B12), 
\begin{align*}
\sqrt n (\bm{\hat \eta}_n - \bm\eta_0)\Rightarrow 
N(0, {\bm J_m}^{-1}{\bm I_m}{\bm J_m}^{-1}) \quad\text{as $n\to\infty$,}
\end{align*}
where
\begin{align*}
{\bm J_m}:=&{\rm E}\skakko{\frac{\partial^2}{\partial\bm\eta \partial\bm\eta^\top}\ell_m(Z_t, m_t(\bm\eta_0))}\\
=&
{\rm E}\skakko{
\ell_m^{\prime\prime}(Z_t, m_t(\bm\eta_0))
\frac{\partial}{\partial\bm\eta}m_t(\bm\eta_0)
\frac{\partial}{\partial\bm\eta^\top}m_t(\bm\eta_0)}\\
\text{and }
{\bm I_m}:=&{\rm E}\skakko{
\frac{\partial}{\partial\bm\eta }\ell_m(Z_t, m_t(\bm\eta_0))
\frac{\partial}{\partial\bm\eta^\top}\ell_m(Z_t, m_t(\bm\eta_0))
}\\
=&
{\rm E}\skakko{
\skakko{{\ell_m^{\prime}}(Z_t, m_t(\bm\eta_0))}^2
\frac{\partial}{\partial\bm\eta}m_t(\bm\eta_0)
\frac{\partial}{\partial\bm\eta^\top}m_t(\bm\eta_0)}.
\end{align*}
\end{lemma}

Next, we consider the estimation of $\bm \theta_0$. The M-estimator of $\bm \theta_0$ is defined by
\begin{align*}
\bm{\hat \theta}_n
:=\argmax_{\bm\theta\in \bm \Theta}\tilde L_{n,v}(\bm{\hat \eta}_n,\bm\theta),
\end{align*}
where 
$\tilde L_{n,v}(\bm{\hat \eta}_n,\bm\theta):= \frac{1}{n}\sum_{t=1}^n \ell_v(Z_t,\tilde m_t(\bm{\hat\eta}_n),\tilde v_t(\bm\theta))$ and 
$\ell_v$ is a measurable function.

We suppose the following conditions.
\begin{ass}{\rm 
(C1) It holds that
$$\sup_{(\bm \eta,\bm\theta)\in \bm H\times \bm\Theta}
\left|\ell_v(Z_t,\tilde m_t(\bm{\eta}),\tilde v_t(\bm\theta))-\ell_v(Z_t,m_t(\bm{\eta}), v_t(\bm\theta))\right|\to0 \quad \text{a.s. as $t\to\infty$.}$$

\noindent
(C2) The function $\ell_v(\cdot,\cdot,\cdot)$ is differentiable with respect to its second and third components and, for some neighborhood $B(\bm \eta_0)$ of $\bm \eta_0$, 
\begin{align*}
{\rm E}
\skakko{
\sup_{\bm \eta \in B(\bm \eta_0)}
\sup_{\bm \theta \in \Theta}
\left\|
\frac{\partial}{\partial\bm\eta}
\ell_v(Z_t, m_t(\bm\eta),v_t(\bm\theta))
\right\|}<\infty.
\end{align*}

\noindent
(C3) The function $v$ is almost surely continuous with respect to $\bm \theta$ and $\ell_v(\cdot,\cdot,\cdot)$ is almost surely continuous  with respect to its third component.\\

\noindent
(C4) The expectation of $\ell_v(Z_t, m_t(\bm\eta_0),v_t(\bm\theta_0))$ is finite.\\

\noindent
(C5) The expectation of $\ell_v(Z_t, m_t(\bm\eta_0),v_t(\bm\theta))$ has a unique maximum at $\bm\theta_0$.\\

\noindent
(C6) The parameter space $\bm \Theta$ is compact.\\

\noindent
(C7) The function $v$ is twice continuously differentiable with respect to $\bm \theta$, and
$$\frac{\partial^2}{\partial z\partial y}\ell_v(x,y,z)\text{ and } 
\frac{\partial^2}{\partial z\partial z}\ell_v(x,y,z)$$ 
exist and are continuous with respect to their second and third components.\\

\noindent
(C8) It holds that
\begin{align*}
&\sup_{(\bm \eta,\bm \theta)\in \bm H\times \bm \Theta}
\left\|\frac{\partial}{\partial \bm \theta}\ell_v(Z_t,\tilde m_t(\bm{\eta}),\tilde v_t(\bm\theta))-\frac{\partial}{\partial \bm \theta}\ell_v(Z_t, m_t(\bm{\eta}),v_t(\bm\theta))\right\|\\
&=O(t^{-\frac{1}{2}-\delta})\quad\text{a.s. as $t\to\infty$.}
\end{align*}

\noindent
(C9) The true parameter $\bm \theta_0$ belongs to the interior of $\bm \Theta$.\\

\noindent
(C10) There exists neighborhoods $B_1(\bm \eta_0)$ of $\bm \eta_0$ and $B_2(\bm \theta_0)$ of $\bm \theta_0$ such that
\begin{align*}
{\rm E}\skakko{
\sup_{(\bm \eta,\bm \theta) \in B_1(\bm \eta_0)\times B_2(\bm \theta_0)}\left\|
\frac{\partial^2}{\partial\bm\theta\partial\bm\eta^\top}
\ell_v(Z_t, m_t(\bm{\eta}),v_t(\bm\theta))
\right\|}<\infty 
\end{align*}
and
\begin{align*}
{\rm E}\skakko{
\sup_{(\bm \eta,\bm \theta) \in B_1(\bm \eta_0)\times B_2(\bm \theta_0)}\left\|
\frac{\partial^2}{\partial\bm\theta\partial\bm\theta^\top}
\ell_v(Z_t, m_t(\bm{\eta}),v_t(\bm\theta))
\right\|}<\infty.
\end{align*}

\noindent
(C11) It holds that ${\rm E}\skakko{\ell_v^\prime(Z_t, m_t(\bm{\eta}_0),v_t(\bm\theta_0))\mid \mathcal F_{t-1}}=0$  a.s., where $\ell_v^\prime(\cdot,\cdot,\cdot)$ is the first derivative of $\ell_v(\cdot,\cdot,\cdot)$ with respect to its third component.\\

\noindent
(C12) The following holds true: ${\rm E}\skakko{\ell_v^{\prime\prime}(Z_t, m_t(\bm{\eta}_0),v_t(\bm\theta_0))\mid \mathcal F_{t-1}}$ is {almost surely negative}, where $\ell_v^{\prime\prime}(\cdot,\cdot,\cdot)$ is the second derivative of $\ell_v(\cdot,\cdot,\cdot)$ with respect to its third component, and if 
${\bm s}^\top\frac{\partial}{\partial\bm\theta}v_t(\bm\theta_0)=0$, then 
$\bm s=\bm0$.\\

\noindent
(C13) 
It holds that
$${\rm E}\left\|\frac{\partial}{\partial \bm \theta}\ell_v(Z_t, m_t(\bm{\eta}_0), v_t(\bm{\theta}_0))\right\|^{2+\delta}<\infty.$$}
\end{ass}

The following theorem shows that the strong consistency of $\bm{\hat \theta}_n$ and  asymptotic joint normality of $\bm{\hat \eta}_n$ and $\bm{\hat \theta}_n$.
\begin{theorem}\label{thma1}
Under (A), (B1)--(B12), and (C1)--(C6), 
$\bm{\hat \theta}_n$ tends almost surely to $\bm\theta_0$ as $n\to\infty$. Moreover, under (A), (B1)--(B12), and (C1)--(C14), it holds that 
\begin{align*}
\sqrt n
\begin{pmatrix}
\bm{\hat \eta}_n - \bm\eta_0\\
\bm{\hat \theta}_n - \bm\theta_0\\
\end{pmatrix}\Rightarrow
N\skakko{\bm0, \bm V}
 \quad\text{as $n\to\infty$,}
\end{align*}
where $\bm V:=(\bm v_{ij})_{i,j=1,2}$ and
\begin{align*}
\bm v_{11}:=&{\bm J_m^{-1}}{\bm I_m}{\bm J_m^{-1}},\\
\bm v_{12}:=&-{\bm J_m^{-1}}{\bm I_m}{\bm J_m^{-1}}{\bm J_{vm}^\top}{\bm J_{v}^{-1}}
+{\bm J_{m}^{-1}}{\bm I_{mv}}{\bm J_{v}^{-1}}\\
=&{\bm J_{m}^{-1}}({\bm I_{mv}}-{\bm I_m}{\bm J_m^{-1}}{\bm J_{vm}^\top}){\bm J_{v}^{-1}},\\
\bm v_{21}:=&-{\bm J_{v}^{-1}}{\bm J_{vm}}{\bm J_{m}^{-1}}{\bm I_{m}}{\bm J_{m}^{-1}}+{\bm J_{v}^{-1}}{\bm I_{mv}^{\top}}{\bm J_{m}^{-1}}\\
=&\bm v_{12}^\top,\\
\bm v_{22}:=&{\bm J_{v}^{-1}}{\bm J_{vm}}{\bm J_{m}^{-1}}{\bm I_{m}}{\bm J_{m}^{-1}}{\bm J_{vm}^{\top}}{\bm J_{v}^{-1}}
-{\bm J_{v}^{-1}}{\bm I_{mv}^{\top}}{\bm J_{m}^{-1}}{\bm J_{vm}^{\top}}{\bm J_{v}^{-1}}
\\
&-{\bm J_{v}^{-1}}{\bm J_{vm}}{\bm J_{m}^{-1}}{\bm I_{mv}}{\bm J_{v}^{-1}}+{\bm J_{v}^{-1}}{\bm I_{v}}{\bm J_{v}^{-1}}\\
=&
{\bm J_{v}^{-1}}(
{\bm I_{v}}
+{\bm J_{vm}}{\bm J_{m}^{-1}}{\bm I_{m}}{\bm J_{m}^{-1}}{\bm J_{vm}^{\top}}
-{\bm I_{mv}^{\top}}{\bm J_{m}^{-1}}{\bm J_{vm}^{\top}}
-{\bm J_{vm}}{\bm J_{m}^{-1}}{\bm I_{mv}}
){\bm J_{v}^{-1}}
\end{align*}
with
\begin{align*}
\bm{J}_{vm}:=&{\rm E}\skakko{\frac{\partial^2}{\partial \bm \theta \partial \bm \eta^\top}\ell_v(Z_t, m_t(\bm{\eta}_0), v_t(\bm{\theta}_0))},\\
\bm{J}_{v}:=&{\rm E}\skakko{\frac{\partial^2}{\partial \bm \theta \partial \bm \theta^\top}\ell_v(Z_t, m_t(\bm{\eta}_0), v_t(\bm{\theta}_0))},\\
\bm I_{mv}:=&{\rm E}\skakko{\frac{\partial}{\partial \bm \eta}\ell_m(Z_t, m_t(\bm{\eta}_0))\frac{\partial}{\partial \bm \theta^\top}\ell_v(Z_t, m_t(\bm{\eta}_0), v_t(\bm{\theta}_0))},\\
\text{and }\bm I_v :=&{\rm E}\skakko{\frac{\partial}{\partial \bm \theta}\ell_v(Z_t, m_t(\bm{\eta}_0), v_t(\bm{\theta}_0))\frac{\partial}{\partial \bm \theta^\top}\ell_v(Z_t, m_t(\bm{\eta}_0), v_t(\bm{\theta}_0))}.
\end{align*}

\end{theorem}

\section{Proof of Theorem \ref{thma1}}\label{suppl2}
\setcounter{equation}{0}
In this section, we prove Theorem \ref{thma1}. We follow the proof of \citet[Theorems 2.1 and 2.2]{af16} and \citet[Theorem 7.1, p.159]{fz10}.

First, we prove the strong consistency of $\bm{\hat \theta}_n$. Assumption (C1) gives that the effect of initial value is asymptotically negligible, that is,
\begin{align*}
&\sup_{\bm\theta\in \bm \Theta}
\left|\tilde L_{n,v}(\bm{\hat \eta}_n,\bm\theta) - L_{n,v}(\bm{\hat \eta}_n,\bm\theta)\right|\\
\leq 
&\frac{1}{n}\sum_{t=1}^n\sup_{(\bm \eta,\bm\theta)\in \bm H\times \bm\Theta}
\left|\ell_v(Z_t,\tilde m_t(\bm{\eta}),\tilde v_t(\bm\theta))-\ell_v(Z_t,m_t(\bm{\eta}), v_t(\bm\theta))\right|\to0 \quad \text{a.s.}
\end{align*}
as $n\to\infty$.
Assumptions (B8) and (C2) and the strong consistency of $\bm\eta_0$ yield that
\begin{align*}
&\sup_{\bm\theta\in \bm \Theta}
\left|L_{n,v}(\bm{\hat \eta}_n,\bm\theta) - L_{n,v}(\bm{ \eta}_0,\bm\theta)\right|\\
\leq 
&\frac{1}{n}\sum_{t=1}^n\sup_{\bm\theta\in \bm \Theta}\left|\ell_v(Z_t, m_t(\bm{\hat \eta}_n),v_t(\bm{\hat \eta}_n,\bm\theta))-\ell_v(Z_t,m_t(\bm{\eta}_0), v_t(\bm\theta))\right|\\
\leq 
&\left\|\bm{\hat \eta}_n-\bm{ \eta}_0\right\|_{\ell_1}
\frac{1}{n}\sum_{t=1}^n\sup_{\bm\theta\in \bm \Theta}\left\|
\frac{\partial}{\partial \bm\eta}\ell_v(Z_t, m_t(\bm{\hat \eta}_n^*),v_t(\bm\theta))\right\|_{\ell_1}
\to0 \quad \text{a.s. as $n\to\infty$,}
\end{align*}
where $\bm{\eta}_0\lesseqgtr\bm{\hat \eta}_n^*\lesseqgtr\bm{\hat \eta}_n$. Now, we have, for any ${\bm \theta}_1$,
\begin{align*}
&\sup_{\bm \theta \in B(\bm \theta_1;r)}\tilde L_{n,v}(\bm{\hat \eta}_n,\bm\theta)\\
=&
\sup_{\bm \theta \in B(\bm \theta_1;r)}\tilde L_{n,v}(\bm{\hat \eta}_n,\bm\theta)
-\sup_{\bm \theta \in B(\bm \theta_1;r)}L_{n,v}(\bm{\hat \eta}_n,\bm\theta)
+\sup_{\bm \theta \in B(\bm \theta_1;r)}L_{n,v}(\bm{\hat \eta}_n,\bm\theta)\\
&-\sup_{\bm \theta \in B(\bm \theta_1;r)}L_{n,v}(\bm{ \eta}_0,\bm\theta)
+\sup_{\bm \theta \in B(\bm \theta_1;r)}L_{n,v}(\bm{ \eta}_0,\bm\theta),
\end{align*}
where $B({\bm \theta}_1;r):=\{{\bm \theta}:\|{\bm \theta}-{\bm \theta}_1 \|_{\ell_1}<1/r\}$, and so by the above argument and the ergodic theorem for non-integrable processes (see \citealp[p.181, Problem 7.3]{fz10}), we observe
\begin{align*}
\limsup_{n\to\infty}\sup_{\bm \theta \in B(\bm \theta_1;r)}\tilde L_{n,v}(\bm{\hat \eta}_n,\bm\theta)
\leq&
{\rm E}\skakko{\sup_{\bm \theta \in B(\bm \theta_1;r)}\ell_v(Z_t, m_t(\bm\eta_0),v_t(\bm\theta))}.
\end{align*}
Here we used the inequality 
$$\sup_{\bm \theta \in B(\bm \theta_1;r)}L_{n,v}(\bm{ \eta}_0,\bm\theta)\leq \frac{1}{n}\sum_{t=1}^n\sup_{\bm \theta \in B(\bm \theta_1;r)}\ell_v(Z_t, m_t(\bm\eta_0),v_t(\bm\theta)).$$
 The Beppo-Levi theorem and Assumption (C3) yield that 
$${\rm E}\skakko{\sup_{\bm \theta \in B(\bm \theta_1;r)}\ell_v(Z_t, m_t(\bm\eta_0),v_t(\bm\theta))}\to
{\rm E}\skakko{\ell_v(Z_t, m_t(\bm\eta_0),v_t(\bm\theta_1))}
\quad \text{as $r\to\infty$}.
$$
Therefore, for any $\bm\theta_1(\neq\bm\theta_0)$, there exists a neighborhood $B(\bm \theta_1)$ satisfying
\begin{align*}
\limsup_{n\to\infty}\sup_{\bm \theta \in B(\bm \theta_1)}\tilde L_{n,v}(\bm{\hat \eta}_n,\bm\theta)
&\leq
{\rm E}\skakko{\ell_v(Z_t, m_t(\bm\eta_0),v_t(\bm{ \eta}_0,\bm\theta_1))}\quad{a.s.}\\
&<
{\rm E}\skakko{\ell_v(Z_t, m_t(\bm\eta_0),v_t(\bm\theta_0))}\quad{a.s.}\\
&=\lim_{n\to\infty} \tilde L_{n,v}(\bm{\hat \eta}_n,\bm\theta_0)\quad{a.s.}
\end{align*}
Here we used Assumption (C5).  
By Assumption (C6), there exists, for any covering set $\{B(\bm \theta_i); i\in\{0\}\cup\mathbb N\}$ of $\Theta$ such that $\bm \theta_i\in\Theta\backslash B(\bm \theta_0)$ for all $i\in\mathbb N$ and $$\limsup_{n\to\infty}\sup_{\bm \theta \in B(\bm \theta_1)}\tilde L_{n,v}(\bm{\hat \eta}_n,\bm\theta)<\lim_{n\to\infty} \tilde L_{n,v}(\bm{\hat \eta}_n,\bm\theta_0),$$
the finite covering set $\{B^\prime(\bm \theta_i);i=0,1,\ldots,s\}$ such that $\Theta:=\cup_{i=0}^sB^\prime(\bm \theta_i)$ and $B^\prime(\bm \theta_0)=B(\bm \theta_0)$. Then, we observe that, for sufficiently large $n$,
\begin{align*}
\sup_{\bm \theta \in \Theta}\tilde L_{n,v}(\bm{\hat \eta}_n,\bm\theta)
=&\max_{i=0,1,\ldots,s}\sup_{\bm \theta \in B^\prime(\bm \theta_i)}\tilde L_{n,v}(\bm{\hat \eta}_n,\bm\theta)\\
=&\sup_{\bm \theta \in B^\prime(\bm \theta_0)}\tilde L_{n,v}(\bm{\hat \eta}_n,\bm\theta)\quad \text{a.s.}
\end{align*}
Since $B(\bm \theta_0)$ can be chosen an arbitrary small set, we obtain the strong consistency of $\bm{\hat \theta}_n$.

Next, we show the asymptotic normality of $(\bm{\hat \eta}_n^\top,\bm{\hat \theta}_n^\top)^\top$. From Assumptions (B6), (B8), (C7), (C8), and (C9), we have
\begin{align}\nonumber
\bm{0}
=&\frac{1}{\sqrt n}\sum_{t=1}^n \frac{\partial}{\partial \bm \theta}\ell_v(Z_t,\tilde m_t(\bm{\hat\eta}_n),\tilde v_t(\bm{\hat \theta}_n))\\\nonumber
=&\frac{1}{\sqrt n}\sum_{t=1}^n\frac{\partial}{\partial \bm \theta} \ell_v(Z_t, m_t(\bm{\hat\eta}_n), v_t(\bm{\hat \theta}_n))+o_p(1)\\\nonumber
=&
\frac{1}{\sqrt n}\sum_{t=1}^n \frac{\partial}{\partial \bm \theta}\ell_v(Z_t, m_t(\bm{\eta}_0), v_t(\bm{\theta}_0))\\\nonumber
&+\frac{1}{ n}\sum_{t=1}^n \frac{\partial^2}{\partial \bm \theta \partial \bm \eta^\top}\ell_v(Z_t, m_t(\bm{\hat\eta}_n^{**}), v_t(\bm{\hat\theta}_n^{**}))\sqrt n(\bm{\hat \eta}_n-\bm{ \eta}_0)\\\nonumber
&+\frac{1}{n}\sum_{t=1}^n \frac{\partial^2}{\partial \bm \theta \partial \bm \theta^\top}\ell_v(Z_t, m_t(\bm{\hat\eta}_n^{**}), v_t(\bm{\hat\theta}_n^{**}))\sqrt n(\bm{\hat \theta}_n-\bm{ \theta}_0)+o_p(1)\\\label{taylor_exp}
&=\frac{1}{\sqrt n}\sum_{t=1}^n \frac{\partial}{\partial \bm \theta}\ell_v(Z_t, m_t(\bm{\eta}_0), v_t(\bm{\theta}_0))+\bm{J}_{vm}^*\sqrt n(\bm{\hat \eta}_n-\bm{ \eta}_0)\\\label{taylor_exp2}
&\quad+\bm{J}_{v}^*\sqrt n(\bm{\hat \theta}_n-\bm{ \theta}_0)+o_p(1),
\end{align}
where $\bm{\eta}_0\lesseqgtr\bm{\hat \eta}_n^{**}\lesseqgtr\bm{\hat \eta}_n$, $\bm{\theta}_0\lesseqgtr\bm{\hat \eta}_n^{**}\lesseqgtr\bm{\hat \theta}_n$, 
$$\bm{J}^*_{vm}:=\frac{1}{ n}\sum_{t=1}^n \frac{\partial^2}{\partial \bm \theta \partial \bm \eta^\top}\ell_v(Z_t, m_t(\bm{\hat\eta}_n^{**}), v_t(\bm{\hat\theta}_n^{**})),$$ and 
$$\bm{J}_{v}^*:=\frac{1}{n}\sum_{t=1}^n \frac{\partial^2}{\partial \bm \theta \partial \bm \theta^\top}\ell_v(Z_t, m_t(\bm{\hat\eta}_n^{**}), v_t(\bm{\hat\theta}_n^{**})).$$
The assumption (C10) gives that, for some $r$ such that $B(\bm \eta_0;r)\times B(\bm \theta_0;r)\subset B_1(\bm \eta_0)\times B_2(\bm \theta_0)$ and for large $n$ such that  $\bm{\hat \eta}_n^{**}\times\bm{\hat \theta}_n^{**}\in B(\bm \eta_0;r)\times B(\bm \theta_0;r)$,
\begin{align*}
&\left|\frac{1}{ n}\sum_{t=1}^n \frac{\partial^2}{\partial \theta_{ij} \partial \eta_{ij}^\top}\ell_v(Z_t, m_t(\bm{\hat\eta}_n^{**}), v_t(\bm{\hat\theta}_n^{**}))-
{\rm E}\frac{\partial^2}{\partial \theta_{ij} \partial \eta_{ij}^\top}\ell_v(Z_t, m_t(\bm{\eta}_0), v_t(\bm{\theta}_0))\right|\\
\leq&
\frac{1}{ n}\sum_{t=1}^n\sup_{(\bm \eta,\bm \theta) \in B(\bm \eta_0;r)\times B(\bm \theta_0;r)}\\
&\left| \frac{\partial^2}{\partial \theta_{ij} \partial \eta_{ij}^\top}\ell_v(Z_t, m_t(\bm{\eta}), v_t(\bm{\theta}))-
{\rm E}\frac{\partial^2}{\partial \theta_{ij} \partial \eta_{ij}^\top}\ell_v(Z_t, m_t(\bm{\eta}_0), v_t(\bm{\theta}_0))\right|\\
\to&
{\rm E}\sup_{(\bm \eta,\bm \theta) \in B(\bm \eta_0;r)\times B(\bm \theta_0;r)}\\
&\left| \frac{\partial^2}{\partial \theta_{ij} \partial \eta_{ij}^\top}\ell_v(Z_t, m_t(\bm{\eta}), v_t(\bm{\theta}))-
{\rm E}\frac{\partial^2}{\partial \theta_{ij} \partial \eta_{ij}^\top}\ell_v(Z_t, m_t(\bm{\eta}_0), v_t(\bm{\theta}_0))\right|
\end{align*}
a.s. as $n\to\infty$. By applying the Beppo-Levi theorem, we have
\begin{align*}
&{\rm E}\sup_{(\bm \eta,\bm \theta) \in B(\bm \eta_0;r)\times B(\bm \theta_0;r)}\\
&\left| \frac{\partial^2}{\partial \theta_{ij} \partial \eta_{ij}^\top}\ell_v(Z_t, m_t(\bm{\eta}), v_t(\bm{\theta}))-
{\rm E}\frac{\partial^2}{\partial \theta_{ij} \partial \eta_{ij}^\top}\ell_v(Z_t, m_t(\bm{\eta}_0), v_t(\bm{\theta}_0))\right|
\to 0
\end{align*}
a.s. as $r\to\infty$, and thus, $\bm{J}_{vm}^*$ converges almost surely to $\bm{J}_{vm}$ as $n\to\infty$, where
$$\bm{J}_{vm}:={\rm E}\skakko{\frac{\partial^2}{\partial \bm \theta \partial \bm \eta^\top}\ell_v(Z_t, m_t(\bm{\eta}_0), v_t(\bm{\theta}_0))}.$$
 Similarly, we obtain $\bm{J}_{v}^*$ converges almost surely to $\bm{J}_{v}$ as $n\to\infty$, where
\begin{align*}
\bm{J}_{v}:={\rm E}\frac{\partial^2}{\partial \bm \theta \partial \bm \theta^\top}\ell_v(Z_t, m_t(\bm{\eta}_0), v_t(\bm{\theta}_0)).
\end{align*}
The non-singularity of  $\bm{J}_{v}$ is ensured by Assumptions  (C12). Actually, we observe
\begin{align*}
\bm{J}_{v}:=&{\rm E}\skakko{\frac{\partial^2}{\partial \bm \theta \partial \bm \theta^\top}\ell_v(Z_t, m_t(\bm{\eta}_0), v_t(\bm{\theta}_0))}\\
=&{\rm E}\lkakko{
\frac{\partial}{\partial \bm \theta^\top}\mkakko{
\skakko{\frac{\partial}{\partial \bm \theta}v_t(\bm{\theta}_0)}
\ell_v^\prime(Z_t, m_t(\bm{\eta}_0), v_t(\bm{\theta}_0))
}}\\
=&{\rm E}\left\{
\skakko{\frac{\partial}{\partial \bm \theta}v_t(\bm{\theta}_0)}
\skakko{\frac{\partial}{\partial \bm \theta^\top}v_t(\bm{\theta}_0)}
\ell_v^{\prime\prime}(Z_t, m_t(\bm{\eta}_0), v_t(\bm{\theta}_0))\right.\\
&\left.+
\skakko{\frac{\partial^2}{\partial \bm \theta \partial \bm \theta^\top}v_t(\bm{\theta}_0)}
\ell_v^\prime(Z_t, m_t(\bm{\eta}_0), v_t(\bm{\theta}_0))
\right\}\\
=&
{\rm E}\mkakko{
\skakko{\frac{\partial}{\partial \bm \theta}v_t(\bm{\theta}_0)}
\skakko{\frac{\partial}{\partial \bm \theta^\top}v_t(\bm{\theta}_0)}
\ell_v^{\prime\prime}(Z_t, m_t(\bm{\eta}_0), v_t(\bm{\theta}_0))}.
\end{align*}
For any $\bm s\in\mathbb R^{d^\prime}$, it holds 
\begin{align*}
{\bm s}^\top \bm{J}_{v}{\bm s}
=&{\rm E}\mkakko{
\left|{\bm s}^\top\frac{\partial}{\partial \bm \theta}v_t(\bm{\theta}_0)\right|^2
{\rm E}\skakko{\ell_v^{\prime\prime}(Z_t, m_t(\bm{\eta}_0), v_t(\bm{\theta}_0))\mid \mathcal F_{t-1}}}\geq0,
\end{align*}
which shows that $\bm{J}_{v}$ is a non-negative definite matrix.
Suppose ${\bm s}^\top \bm{J}_{v}{\bm s}=0$. Then, we have 
${\bm s}^\top\frac{\partial}{\partial \bm \theta}v_t(\bm{\theta}_0)=0$, which by Assumption (C11) and (C12), gives ${\bm s}=\bm 0$. This implies that $\bm{J}_{v}$ is a positive definite matrix. 
Similarly, $\bm{J}_{m}$ is a positive definite matrix by Assumption  (B11).
From \eqref{taylor_exp} and \eqref{taylor_exp2}, it follows that
\begin{align*}
&\begin{pmatrix}
\sqrt n(\bm{\hat \eta}_n-\bm{ \eta}_0)\\
\sqrt n(\bm{\hat \theta}_n-\bm{ \theta}_0)
\end{pmatrix}\\
&=
\begin{pmatrix}
\sqrt n(\bm{\hat \eta}_n-\bm{ \eta}_0)\\
-{\bm{J}_{v}^*}^{-1}\frac{1}{\sqrt n}\sum_{t=1}^n \frac{\partial}{\partial \bm \theta}\ell_v(Z_t, m_t(\bm{\eta}_0), v_t(\bm{\theta}_0))
-{\bm{J}_{v}^*}^{-1}\bm{J}_{vm}^*\sqrt n(\bm{\hat \eta}_n-\bm{ \eta}_0)
\end{pmatrix}\\
&\quad+o_p(1)\\
&=
\begin{pmatrix}
\bm{\mathcal  I}_p&\bm {\mathcal O}_p\\
-{\bm{J}_{v}^*}^{-1}\bm{J}_{vm}^*&-{\bm{J}_{v}^*}^{-1}
\end{pmatrix}
\begin{pmatrix}
\sqrt n(\bm{\hat \eta}_n-\bm{ \eta}_0)\\
\frac{1}{\sqrt n}\sum_{t=1}^n \frac{\partial}{\partial \bm \theta}\ell_v(Z_t, m_t(\bm{\eta}_0), v_t(\bm{\theta}_0))
\end{pmatrix}+o_p(1)\\
&=
\begin{pmatrix}
\bm{\mathcal  I}_p&\bm {\mathcal O}_p\\
-\bm{J}_{v}^{-1}\bm{J}_{vm}&-\bm{J}_{v}^{-1}
\end{pmatrix}
\begin{pmatrix}
-\bm J_m^{-1}\frac{1}{\sqrt n}\sum_{t=1}^n \frac{\partial}{\partial \bm \eta}\ell_m(Z_t, m_t(\bm{\eta}_0))
\\
\frac{1}{\sqrt n}\sum_{t=1}^n \frac{\partial}{\partial \bm \theta}\ell_v(Z_t, m_t(\bm{\eta}_0), v_t(\bm{\theta}_0))
\end{pmatrix}+o_p(1),
\end{align*}
where $\bm{\mathcal  I}_p$ is an identity matrix of size $p$ and $\bm {\mathcal O}_p$ is a $p$-by-$p$ zero matrix.
The proof will be complete by
\begin{align*}
&\begin{pmatrix}
\bm{\mathcal  I}_p&\bm {\mathcal O}_p\\
-\bm{J}_{v}^{-1}\bm{J}_{vm}&-\bm{J}_{v}^{-1}
\end{pmatrix}
\begin{pmatrix}
\bm J_m^{-1}\bm I_m\bm J_m^{-1}&-\bm J_m^{-1}\bm I_{mv}\\
-\bm I_{mv}^{\top}\bm J_m^{-1}&\bm I_v\\
\end{pmatrix}
\begin{pmatrix}
\bm{\mathcal  I}_p&\bm {\mathcal O}_p\\
-\bm{ J}_{v}^{-1}\bm{J}_{vm}&-\bm{ J}_{v}^{-1}
\end{pmatrix}^\top\\
&=\bm V,
\end{align*}
provided we can show that
\begin{align}\label{CLT}
&\begin{pmatrix}
\bm J_m^{-1}\frac{1}{\sqrt n}\sum_{t=1}^n \frac{\partial}{\partial \bm \eta}\ell_m(Z_t, m_t(\bm{\eta}_0))
\\
\frac{1}{\sqrt n}\sum_{t=1}^n \frac{\partial}{\partial \bm \theta}\ell_v(Z_t, m_t(\bm{\eta}_0), v_t(\bm{\theta}_0))
\end{pmatrix}\\\label{CLT2}
&\quad\Rightarrow
N\skakko{\bm 0_{d+d^\prime},
\begin{pmatrix}
\bm J_m^{-1}\bm I_m\bm J_m^{-1}&-\bm J_m^{-1}\bm I_{mv}\\
-\bm I_{mv}^{\top}\bm J_m^{-1}&\bm I_v\\
\end{pmatrix}
}
\end{align}
as $n\to\infty$, where
\begin{align*}
\bm I_{mv}:=&{\rm E}\skakko{\frac{\partial}{\partial \bm \eta}\ell_m(Z_t, m_t(\bm{\eta}_0))\frac{\partial}{\partial \bm \theta^\top}\ell_v(Z_t, m_t(\bm{\eta}_0), v_t(\bm{\theta}_0))}\\
\text{and }
\bm I_v :=&{\rm E}\skakko{\frac{\partial}{\partial \bm \theta}\ell_v(Z_t, m_t(\bm{\eta}_0), v_t(\bm{\theta}_0))\frac{\partial}{\partial \bm \theta^\top}\ell_v(Z_t, m_t(\bm{\eta}_0), v_t(\bm{\theta}_0))}.
\end{align*}
From Assumptions (B12) and (C11), we know that 
\begin{align*}
\begin{pmatrix}
\bm J_m^{-1}\frac{1}{\sqrt n}\sum_{t=1}^n \frac{\partial}{\partial \bm \eta}\ell_m(Z_t, m_t(\bm{\eta}_0))
\\
\frac{1}{\sqrt n}\sum_{t=1}^n \frac{\partial}{\partial \bm \theta}\ell_v(Z_t, m_t(\bm{\eta}_0), v_t(\bm{\theta}_0))
\end{pmatrix}
\end{align*}
is a martingale difference sequence. Assumptions (B12) and (C13) ensures the Lindeberg condition: for any constant $(\bm c_1^\top,\bm c_2^\top)^\top\in\mathbb R^{d+d^\prime}$,
\begin{align*}
&\frac{1}{\sqrt n}\sum_{t=1}^n{\rm E}\Bigg[\bigg(c_1^\top\bm J_m^{-1}\sum_{t=1}^n \frac{\partial}{\partial \bm \eta}\ell_m(Z_t, m_t(\bm{\eta}_0))\\
&\quad\quad\quad\quad\quad\quad+
c_2^\top\sum_{t=1}^n \frac{\partial}{\partial \bm \theta}\ell_v(Z_t, m_t(\bm{\eta}_0), v_t(\bm{\theta}_0))\bigg)^2\\
&\times\mathbb I\Bigg\{\frac{1}{\sqrt n}\Bigg|c_1^\top\bm J_m^{-1}\sum_{t=1}^n \frac{\partial}{\partial \bm \eta}\ell_m(Z_t, m_t(\bm{\eta}_0))\\
&\quad\quad\quad\quad\quad\quad+
c_2^\top\sum_{t=1}^n \frac{\partial}{\partial \bm \theta}\ell_v(Z_t, m_t(\bm{\eta}_0), v_t(\bm{\theta}_0))\Bigg|>\epsilon\Bigg\}\Bigg]\\
&=
{\rm E}\Bigg[\bigg(c_1^\top\bm J_m^{-1}\sum_{t=1}^n \frac{\partial}{\partial \bm \eta}\ell_m(Z_t, m_t(\bm{\eta}_0))+
c_2^\top\sum_{t=1}^n \frac{\partial}{\partial \bm \theta}\ell_v(Z_t, m_t(\bm{\eta}_0), v_t(\bm{\theta}_0))\bigg)^2\\
&\times\mathbb I\Bigg\{\frac{1}{(\epsilon\sqrt n)^\delta}\Bigg|c_1^\top\bm J_m^{-1}\sum_{t=1}^n \frac{\partial}{\partial \bm \eta}\ell_m(Z_t, m_t(\bm{\eta}_0))\\
&\quad\quad\quad\quad\quad\quad+
c_2^\top\sum_{t=1}^n \frac{\partial}{\partial \bm \theta}\ell_v(Z_t, m_t(\bm{\eta}_0), v_t(\bm{\theta}_0))\Bigg|^\delta>1\Bigg\}\Bigg],\\
&\frac{1}{(\epsilon\sqrt n)^\delta}{\rm E}\Bigg|c_1^\top\bm J_m^{-1}\sum_{t=1}^n \frac{\partial}{\partial \bm \eta}\ell_m(Z_t, m_t(\bm{\eta}_0))
+c_2^\top\sum_{t=1}^n \frac{\partial}{\partial \bm \theta}\ell_v(Z_t, m_t(\bm{\eta}_0), v_t(\bm{\theta}_0))\Bigg|^{2+\delta},
\end{align*}
which tends to zero as $n\to\infty$.

The ergodic theorem yields that 
\begin{align*}
&\frac{1}{\sqrt n}\sum_{t=1}^n{\rm E}\Bigg\{\bigg(c_1^\top\bm J_m^{-1}\sum_{t=1}^n \frac{\partial}{\partial \bm \eta}\ell_m(Z_t, m_t(\bm{\eta}_0))\\
&\quad\quad\quad\quad\quad\quad+
c_2^\top\sum_{t=1}^n \frac{\partial}{\partial \bm \theta}\ell_v(Z_t, m_t(\bm{\eta}_0), v_t(\bm{\theta}_0))\bigg)^2\mid \mathcal F_{t-1}\Bigg\}\\
&\to
{\rm E}\bigg(c_1^\top\bm J_m^{-1}\sum_{t=1}^n \frac{\partial}{\partial \bm \eta}\ell_m(Z_t, m_t(\bm{\eta}_0))+
c_2^\top\sum_{t=1}^n \frac{\partial}{\partial \bm \theta}\ell_v(Z_t, m_t(\bm{\eta}_0), v_t(\bm{\theta}_0))\bigg)^2
\end{align*}
as $n\to\infty$.
Therefore, the Cramer--Wold device and the martingale central limit theorem yield \eqref{CLT} and \eqref{CLT2}.

\section*{Acknowledgements}
Authors would like to express sincere thanks to anonymous reviewers, which greatly improved the earlier version of this paper.

\bibliographystyle{imsart-nameyear}
\bibliography{ref}

\end{document}